\documentclass[12pt]{amsart}

\usepackage{amsmath, amssymb}
\usepackage{amscd}
\usepackage{verbatim}
\usepackage{cleveref}
\usepackage{tikz}
\usetikzlibrary{arrows,chains,matrix,positioning,scopes}

\makeatletter
\tikzset{join/.code=\tikzset{after node path={%
\ifx\tikzchainprevious\pgfutil@empty\else(\tikzchainprevious)%
edge[every join]#1(\tikzchaincurrent)\fi}}}
\makeatother

\tikzset{>=stealth',every on chain/.append style={join},
         every join/.style={->}}

\newtheorem{definition}{Definition}[section]
\newtheorem{theorem}[definition]{Theorem}
\newtheorem{lemma}[definition]{Lemma}

\newtheorem{proposition}[definition]{Proposition}
\theoremstyle{definition}
\newtheorem{remark}[definition]{Remark}

\newtheorem{example}[definition]{Example}

\newcommand{\noi}{\noindent}

\newcommand{\F}{\mathcal{F}}
\newcommand{\ep}{\varepsilon}
\newcommand{\ra}{\rightarrow}

\begin{document}

\title[Fixed Points and Continuity of Semihypergroup Actions]{Fixed Points and Continuity of Semihypergroup Actions}

\author[C.~Bandyopadhyay]{Choiti Bandyopadhyay}

\address{Department of Mathematical and Statistical Sciences, University of Alberta, Canada}
\address{Current Address: Department of Mathematics, Indian Institute of Technology Kanpur, India}

\email{choiti@ualberta.ca, choitib@iitk.ac.in}

\thanks{Part of this work is included in the author's PhD thesis at the University of Alberta}

\begin{abstract}
In a couple of previous papers \cite{CB1, CB2}, we initiated a systematic study of semihypergroups and had a thorough discussion on certain analytic and algebraic aspects associated to this class of objects. In this article, we introduce and examine (separately) continuous actions on the category of semihypergroups. In particular, we discuss the continuity properties of such actions and explore the equivalence relations between different fixed-point properties of certain actions and the existence of  left-invariant mean(s) on the space of almost periodic functions on a semihypergroup.
\end{abstract}

\maketitle

\section{Introduction}
\label{intro}

\quad In the classical setting of topological semigroups and groups, fixed-point properties of their respective representations on certain subsets of a Banach space  or  a locally convex topological space, have been a prominent area of research since its inception. In particular, the amenability properties of a (semi)topological semigroup are  intrinsically  related to the existence of  fixed points of certain actions of the semigroup on certain spaces, and in fact, one can completely  characterize amenability of different function-spaces of a semigroup using such fixed-point properties (see \cite{DA, LT, LZ, MIT, PA, RI} for example, among others). In this article, we initiate the study of semihypergroup actions  and pursue similar results in the more general setting of semihypergroups, in order to thereby realise where and why this theory deviates from the classical theory of locally compact semigroups. 

\quad A semihypergroup, as one would expect, can be perceived simply as a generalization of a locally compact semigroup, where the product of two points is a certain compact set, rather than a single point. In a nutshell, a semihypergroup is essentially a locally compact topological space where the measure space is equipped with a certain convolution product, turning it into an associative algebra (which unlike hypergroups, need not have an identity or an involution).

\quad The concept of semihypergroups, first introduced as `semiconvos' by Jewett\cite{JE}, arises naturally in abstract harmonic analysis in terms of left(right) coset spaces of locally compact groups, and orbit spaces of affine group-actions which appear frequently in different areas of research in mathematics, including Lie groups, coset theory, homogeneous spaces, dynamical systems, and  ordinary and partial differential equations, to name a few. These objects arising from locally compact groups, although retain some interesting structures from the parent category, often fail to be a topological group, semigroup or even a hypergroup (see \cite{CB1,JE} for detailed reasons). The fact that semihypergroups admit a broader range of examples arising from different fields of research compared to classical semigroup, group and hypergroup theory, and yet sustains enough structure to allow an independent theory to develop, makes it an intriguing and useful area of study with essentially a broader range of applications.
 
%\vspace{0.03in}

\quad However, unlike hypergroups, the category of (topological) semihypergroups still lack an extensive  systematic theory on it. In a couple of previous papers \cite{CB1, CB2} we initiated developing a systematic theory on semihypergroups. Our approach towards it has been to define and explore some basic algebraic and analytic aspects and areas on semihypergroups, which are imperative for further development of any category of analytic objects in general. In the first two installments of the study, we introduced and developed the concepts and theories of homomorphisms, ideals, kernels, almost periodic and weakly almost periodic functions and free products on a semihypergroup. 

\quad In this article, we advance the theory by introducing and investigating actions on the category of semihypergroups. In particular, we introduce and discuss the relation between several possible natural  definitions of semihypergroup actions, investigate their continuity properties, and finally proceed to examine the equivalence between amenability properties on the space of almost periodic functions of a semihypergroup, and the existence of certain natural fixed point properties of semihypergroup actions. We see that the classical equivalences in the category of semigroups,  do not always extend to semihypergroups for a number of obstacles and differences in the product-structure of these two classes of objects, as discussed before. For example, unlike topological semigroups, the spaces of almost periodic functions on a semihypergroup do not necessarily form an algebra. The rest of the article is organized as the following.

\quad In the next, \textit{i.e.}, second section of this article, we recall some preliminary definitions and notations given by Jewett in \cite{JE}, and introduce some new definitions required for further development. We conclude the section with listing some important examples of semihypergroups and hypergroups. 

\quad In the third section, we initiate the study of semihypergroup actions. We first introduce a 
 `\textit{general action}' of a semihypergroup $K$ on a Hausdorff topological space $X$, and investigate  when separate continuity of  such an  action forces joint continuity on the underlying product space. Later we introduce a couple of natural definitions of semihypergroup actions on a locally convex space  and discuss their structural equivalence, as well as their relation to general semihypergroup actions,  and the immediate implications to the continuity conditions.  
 
\quad In the final, \textit{i.e,} fourth section of the article, we investigate the existence and equivalence-relations between two different fixed-point properties and amenability on the space $AP(K)$ of almost periodic functions of a (semitoplogical) semihypergroup. We first show that any continuous affine action of a commutative semihypergroup must have a fixed point. On the other hand, we see that for any semihypergroup $K$, if any continuous affine action of $K$ on a compact convex space has a fixed point, then $AP(K)$ must admit a left-invariant mean. Next we discuss the existing Arens product structure \cite{AR, CB1} on $AP(K)^*$, and thereby provide some necessary and sufficient conditions for $K$ to admit a left invariant mean on $AP(K)$. Now finally, we pursue similar equivalence conditions in terms of the measure algebra $M(K)$ of $K$. We consider a different natural fixed point property in terms of $M(K)$, and thereby acquire a complete characterization of amenability properties on  $AP(K)$.  %We show that  important lemmas regarding the weak$^*$-density properties of point-evaluation maps on function spaces of locally compact Hausdorff spaces, in general. 

%\vspace{0.03in}

%Your text comes here. Separate text sections with

%%%%%%%%%%%%%%%%%%%%%%%%%%%%%%%%%%%%%%%%%%%%%%%%%%%%%%%%%%%%%%%%%%%%%%%%%%%%%%%%%%%%%%%%%%%%%%%%%%%%%%%%%%%%%%%%%%%%%%%%%%%%%%%%%%%%%%%%%%%%%%%%
%%%%%%%%%%%%%%%%%%%%%%%%%%%%%%%%%%%%%%%%%%%%%%%%%         PRELIMINARY         %%%%%%%%%%%%%%%%%%%%%%%%%%%%%%%%%%%%%%%%%%%%%%%%%%%%%%%%%%%%%%%%%%
%%%%%%%%%%%%%%%%%%%%%%%%%%%%%%%%%%%%%%%%%%%%%%%%%%%%%%%%%%%%%%%%%%%%%%%%%%%%%%%%%%%%%%%%%%%%%%%%%%%%%%%%%%%%%%%%%%%%%%%%%%%%%%%%%%%%%%%%%%%%%%%%
%%%%%%%%%%%%%%%%%%%%%%%%%%%%%%%%%%%%%%%%%%%%%%%%%%%%%%%%%%%%%%%%%%%%%%%%%%%%%%%%%%%%%%%%%%%%%%%%%%%%%%%%%%%%%%%%%%%%%%%%%%%%%%%%%%%%%%%%%%%%%%%%

\section{Preliminary}
\label{Preliminary}

\noi We first list a preliminary set of notations and definitions that we will use throughout the text. All the topologies throughout this text are assumed to be Hausdorff.%, followed by a brief introduction to the tools and concepts associated to the formal definition of a semihypergroup (referred to as `semiconvo' in \cite{JE}) and some function spaces on it. Finally, we list some well-known natural examples of a semihypergroup \cite{CB1,JE,ZE}. 

%\vspace{0.03in}

\noi For any locally compact Hausdorff topological space $X$, we denote by $M(X)$ the space of all regular complex Borel measures on $X$, where $ M_F^+(X), M^+(X)$ and $P(X)$ respectively denote the subsets of $M(X)$ consisting of all non-negative measures with finite support, all finite non-negative regular Borel measures  and all probability measures on $X$. For any measure $\mu$ on $X$, we denote by $supp(\mu)$ the support of the measure $\mu$. Moreover, $ B(X), C(X), C_c(X)$ and $C_0(X)$ denote the function spaces of all bounded functions, bounded continuous functions, compactly supported continuous functions and continuous functions vanishing at infinity on $X$ respectively.

%\vspace{0.03in}

\noi Unless mentioned otherwise, the space $M^+(X)$ is equipped with the \textit{cone topology}   \cite{JE}, \textit{i.e,} the weak topology on $M^+(X)$ induced by $ C_c^+(X)\cup \{\chi_{_X}\}$. We denote  the set of all compact subsets of $X$ by $\mathfrak{C}(X)$, and consider the natural \textit{Michael topology} \cite{MT} on it, which makes it into a locally compact Hausdorff space.  For any element $x\in X$, we denote by $p_x$ the point-mass measure or the Dirac measure at the point $ x $.

%\vspace{0.03in}

\noi For any three locally compact Hausdorff spaces $X, Y, Z$, a bilinear map $\Psi : M(X) \times M(Y) \rightarrow M(Z)$ is called \textit{positive continuous} if the following properties hold true.
\begin{enumerate}
\item $\Psi(\mu, \nu) \in M^+(Z)$ whenever $(\mu, \nu) \in M^+(X) \times M^+(Y)$.
\item The map $\Psi |_{M^+(X) \times M^+(Y)}$ is continuous.
\end{enumerate}
%\end{definition}

%\vspace{0.03in}

\noi Now we state the formal definition for a (topological) semihypergroup. Note that we follow Jewett's notion \cite{JE} in terms of the definitions and notations, in most cases.

%\vspace {0.06in}

\begin{definition}[Semihypergroup]\label{shyper} A pair $(K,*)$ is called a (topological) semihypergroup if they satisfy the following properties:

%\vspace{0.06in}

\begin{description} % (A) (B) (C)
\item[(A1)] $K$ is a locally compact Hausdorff space and $*$ defines a binary operation on $M(K)$ such that $(M(K), *)$ becomes an associative algebra.

\item[(A2)] The bilinear mapping $* : M(K) \times M(K) \rightarrow M(K)$ is positive continuous.

\item[(A3)] For any $x, y \in K$ the measure $(p_x * p_y)$ is a probability measure with compact support.

\item[(A4)] The map $(x, y) \mapsto \mbox{ supp}(p_x * p_y)$ from $K\times K$ into $\mathfrak{C}(K)$ is continuous.
\end{description}
\end{definition}

%\vspace{0.06in}

\noi Note that for any $A,B \subset K$ the convolution of subsets is defined as the following:
$$A*B := \cup_{x\in A, y\in B} \ supp(p_x*p_y)  .$$

\noi  We define the concepts of left (resp. right) topological and semitopological semihypergroups, in accordance with similar concepts in the classical semigroup theory.

%\vspace{0.06in}

\begin{definition}
A pair $(K, *)$ is called a left (resp. right) topological semihypergroup if it satisfies all the conditions of Definition \ref{shyper}, with property {($A2$)} replaced by property {($A2 '$)} (resp. property {($A2 ''$)}), given as the following:

\begin{description}
\item[(A2$'$)] The map $(\mu, \nu) \mapsto \mu*\nu$ is positive and for each $\omega \in M^+(K)$ the map\\ $L_\omega:M^+(K) \rightarrow M^+(K)$ given by $L_\omega(\mu)= \omega*\mu$ is continuous.\\
\item[(A2$''$)] The map $(\mu, \nu) \mapsto \mu*\nu$ is positive and for each $\omega \in M^+(K)$ the map\\ $R_\omega:M^+(K) \rightarrow M^+(K)$ given by $R_\omega(\mu)= \mu*\omega$ is continuous.
\end{description}
\end{definition}

%\vspace{0.09in}

A pair $(K, *)$ is called a \textit{semitopological semihypergroup} if it is both left and right topological semihypergroup, \textit{i.e}, if the convolution $*$ on $M(K)$ is only separately continuous.  

For any Borel measurable function $f$ on a (semitopological) semihypergroup $K$ and each $x, y \in K$, we define the left translate $L_xf$ of $f$ by $x$ (resp. the right translate $R_yf$ of $f$ by $y$) as
$$ L_xf(y) = R_yf(x) = f(x*y) := \int_K f \ d(p_x*p_y)\ .$$

Unless mentioned otherwise, we will always assume the uniform (supremum) norm $||\cdot ||_u$ on $C(K)$ and $B(K)$. We denote by $\mathcal{B}_1$ the closed unit ball of $C(K)^*$. Similarly, for any linear subspace $\F$ of $C(K)$, we denote the closed unit ball of $\F^*$ as $\mathcal{B}_1(\F^*):= \{\omega \in \F^*: ||\omega|| \leq 1\}$. Moreover, $\F$ is called left (resp. right) translation-invariant if $L_xf\in \F$ (resp. $R_xf\in \F$) for each $x\in K, f\in \F$. We simply say that $\F$ is translation-invariant, if it is both left and right translation-invariant.

A function $f\in C(K)$ is called left (resp. right) uniformly continuous if the map $x\mapsto L_xf$ (resp. $x\mapsto R_xf$) from $K$ to $(C(K), ||\cdot ||_u)$ is continuous. We say that $f$ is \textit{uniformly continuous} if it is both left and right uniformly continuous. The space consisting of  all such functions is denoted by $UC(K)$, which forms a norm-closed linear subspace of $C(K)$. 

The left (resp. right) orbit of a function $f\in C(K)$, denoted as $\mathcal{O}_l(f)$ (resp. $\mathcal{O}_r(f)$), is defined as $\mathcal{O}_l(f) := \{L_xf : x\in K\}$ (resp. $\mathcal{O}_r(f) := \{R_xf : x\in K\}$). A function $f\in C(K)$ is called left (resp. right) almost periodic if we have that $\mathcal{O}_l(f)$ (resp. $\mathcal{O}_r(f))$ is relatively compact in $(C(K), ||\cdot ||_u)$. We showed in a previous work \cite[Corollary 4.4]{CB1} that a function $f$ on $K$ is left almost periodic if and only if it is right almost periodic. Hence we  regard any left or right almost periodic function on $K$ simply as an \textit{almost periodic function}, and denote the space of all almost periodic functions on $K$ as $AP(K)$. We further saw in \cite{CB1} that $AP(K)$ is a norm-closed, conjugate-closed (with respect to complex conjugation), translation-invariant linear subspace of $C(K)$ containing constant functions, such that $AP(K)\subseteq UC(K)$.

Now recall \cite{JE} that for any locally compact Hausdorff space $X$, a map $i : X \ra X$ is called a \textit{topological involution} if $i$ is a homeomorphism and $(i\circ i)(x) = x$ for each $x\in X$. On a semitopological semihypergroup $(K, *)$, a topological involution $i : K \ra K$ given by $i(x):= x^-$ is called a \textit{(semihypergroup) involution} if  $(\mu * \nu)^- = \nu^- * \mu^- $ for any $\mu, \nu \in M(K)$. For any measure $\omega \in M(K)$, we have that $$\omega^-(B) := \omega(B^-)= \omega(i(B)),$$ for any Borel measurable subset $B$ of $K$. As expected, an involution on a semihypergroup is analogous to the inverse function on a semigroup. Hence a semihypergroup with an identity and an involution of the following characteristic is a hypergroup.

\begin{definition}[Hypergroup] A semihypergroup $(H, *)$ is called a hypergroup if it satisfies the following conditions :

\begin{description}
\item[(A5)] There exists an element $e \in H$ such that $p_x * p_e = p_e * p_x = p_x$ for any $x\in H$.
\item[(A6)] There exists an involution $x\mapsto x^-$ on $H$ such that $e \in \mbox{supp} (p_x * p_y)$ if and only if $x=y^-$.
\end{description}
\end{definition}

\noi The element $e$ in the above definition is called the \textit{identity} of $H$. Note that the identity and  involution of a hypergroup are necessarily unique \cite{JE}.

%\vspace{0.03in}

\begin{remark}
Given a Hausdorff topological space $K$, in order to define a continuous bilinear mapping $* : M(K) \times M(K) \rightarrow M(K)$, it suffices to only define the measures $(p_x*p_y)$ for each $x, y \in K$.
 This is true since one can then extend  the convolution `$*$' bilinearly to $M_F^+(K)$. As $M_F^+(K)$ is dense in $M^+(K)$ in the cone topology \cite{JE}, one can further achieve a continuous extension of  `$*$' to $M^+(K)$ and hence to the whole of $M(K)$ using bilinearity.
\end{remark}

Finally, the centre  $Z(K)$ of a (semitopological) semihypergroup $(K, *)$ is defined to be the largest   semigroup included in $(K, *)$. In other words, we have that 
$$Z(K) := \{ x \in K : \mbox{ supp}(p_x*p_y) \mbox{ and  supp}(p_y*p_x) \mbox{ are singleton for any } y \in K\}.$$
Note that it is possible that $Z(K)=\O$, and even if $x\in Z(K)$, then  $  (p_x*p_y)$ and  $(p_y*p_x)$ need not be supported on the same element in $K$, for each $y\in K$. 

But if $K$ is a hypergroup, then we immediately see that $e\in Z(K)$. In fact,   if $K$ is a hypergroup, then the centre $Z(K)$ is indeed the largest group included in $K$, and it can easily be checked \cite{BH} that the following equivalence holds true.
$$ Z(K) := \{ x \in K : p_x * p_{x^-} = p_{x^-} * p_x = p_e\} .$$
Hence the definition for center of a hypergroup\cite{JE} defined as above, coincides with that of the center of a semihypergroup. The following is an example of a hypergroup with a non-trivial centre. 
\vspace{0.03in}

\begin{example}[Zeuner \cite{ZE}]
  Consider the hypergroup $(H, *)$ where $H = [0, 1]$ and the convolution is defined as
$$ p_s*p_t = \frac{p_{|s-t|} + p_{1-|1-s-t|}}{2} .$$
We immediately see that $Z(H) = \{0, 1\}$.

\end{example}

\noi Now we list some well known examples \cite{JE,ZE} of semihypergroups and hypergroups. See \cite[Section 3]{CB1} for details on the constructions as well as the reasons why most of the structures discussed there, although attain a semihypergroup structure, fail to be hypergroups. %Thus we explore how the shortcomings explained in the introduction are overcome by the category of semihypergroups and hypergroups, as well as why most of the structures discussed there, although attain semihypergroup structures, fail to be hypergroups.

%\vspace{0.06in}

\begin{example} \label{extr}

If $(S, \cdot)$ is a locally compact topological semigroup, then $(S, *)$ is a semihypergroup where $p_x*p_y = p_{_{x.y}}$ for any $x, y \in S$. Similarly, if $(G, \cdot)$ is a locally compact topological group with identity $e_G$, then $(G, *)$ is a hypergroup with the same bilinear operation $*$, identity element $e_G$  and the involution on $G$ defined as $x \mapsto x^{-1}$.

Note that $Z(S) =S$, $Z(G) = G$.
\end{example}

\begin{example} \label{ex2}
Take $T = \{e, a, b\}$ and equip it with the discrete topology. Define
\begin{eqnarray*}
p_e*p_a &=& p_a*p_e \ = \ p_a\\
p_e*p_b &=& p_b*p_e \ = \ p_b\\
p_a*p_b &=& p_b*p_a \ = \ z_1p_a + z_2p_b\\
p_a*p_a &=& x_1p_e + x_2p_a + x_3p_b \\
p_b*p_b &=& y_1p_e + y_2p_a + y_3p_b
\end{eqnarray*}

\noi where $x_i, y_i, z_i \in \mathbb{R}$ such that $x_1+x_2+x_3 = y_1+y_2+y_3 = z_1+z_2 = 1$ and $y_1x_3 = z_1x_1$. Then $(T, *)$ is a commutative hypergroup with identity $e$ and the identity function on $T$ taken as involution. In fact, any finite set can be given several (not necessarily equivalent)  semihypergroup and hypergroup structures.
\end{example}

\begin{example} \label{ex3}
Let $G$ be a locally compact group and $H$ be a compact subgroup of $G$ with normalized Haar measure $\mu$. Consider the left coset space $S := G/H = \{xH : x \in G\}$ and the double coset space $K := G/ /H = \{HxH : x \in G\}$ and equip them with the respective quotient topologies. Then $(S, *)$ is a semihypergroup and $(K, *)$ is a hypergroup where the convolutions are given as following for any $x, y \in G$ :
$$p_{_{xH}} * p_{_{yH}} = \int_H p_{_{(xty)H}} \ d\mu(t), \ \ \ \ p_{_{HxH}} * p_{_{HyH}} = \int_H p_{_{H(xty)H}} \ d\mu(t)  .$$
It can be checked \cite{CB1} that the coset spaces $(S, *)$ fail to have a hypergroup structure. 
\end{example}

\begin{example} \label{ex4}
Let $G$ be a locally compact topological group and $H$ be any compact group with normalized Haar measure $\sigma$. 
For any continuous affine action\cite{JE} $\pi$ of $H$ on $G$, consider the orbit space $\mathcal{O} := \{x^H : x \in G\}$, where for each $x\in G$,  $x^H =  \{\pi(h, x): h\in H\}$ is the orbit of $x$ under the action $\pi$. 

%\vspace{0.05cm}

\noi Consider $\mathcal{O}$ with the quotient topology and the following convolution:
$$ p_{_{x^H}} * p_{_{y^H}} := \int_H \int_H p_{_{(\pi(s, x)\pi(t, y))^H}} \ d\sigma(s) d\sigma(t) .$$
\noi Then $(\mathcal{O}, *)$ becomes a semihypergroup. It can be shown \cite{BH, JE} that $(\mathcal{O}, *)$ becomes a hypergroup only if for each $h\in H$, the map $x\mapsto \pi(h, x) : G\ra G$ is an  automorphism.
\end{example}

%%%%%%%%%%%%%%%%%%%%%%%%%%%%%%%%%%%%%%%%%%%%%%%%%%%%%%%%%%%%%%%%%%%%%%%%%%%%%%%%%%%%%%%%%%%%%%%%%%%%%%%%%%%%%%%%%%%%%%%%%%%%%%%%%%%%%%%%%%%%%%%%
%%%%%%%%%%%%%%%%%%%%%%%%%%%%%%%%%%%%%%%%%%%     FREE STRUCTURE ON SEMIHYPERGROUPS    %%%%%%%%%%%%%%%%%%%%%%%%%%%%%%%%%%%%%%%%%%%%%%%%%%%%%%%%%%%
%%%%%%%%%%%%%%%%%%%%%%%%%%%%%%%%%%%%%%%%%%%%%%%%%%%%%%%%%%%%%%%%%%%%%%%%%%%%%%%%%%%%%%%%%%%%%%%%%%%%%%%%%%%%%%%%%%%%%%%%%%%%%%%%%%%%%%%%%%%%%%%%
%%%%%%%%%%%%%%%%%%%%%%%%%%%%%%%%%%%%%%%%%%%%%%%%%%%%%%%%%%%%%%%%%%%%%%%%%%%%%%%%%%%%%%%%%%%%%%%%%%%%%%%%%%%%%%%%%%%%%%%%%%%%%%%%%%%%%%%%%%%%%%%%

\section{General Semihypergroup Action and Continuity}
\label{Actions}

We first provide a general definition of action of a semihypergroup on a Hausdorff topological space, analogous to its definition for the classical case of topological semigroups. Hence in particular, whenever we take the semihypergroup to be a locally compact semigroup  as outlined in Example \ref{extr}, the definition coincides with that of topological semigroups.

%\vspace{0.2in}

\begin{definition}\label{action}
Let $(K, *)$ be a semihypergroup and $X$ be a Hausdorff topological space. A map $\sigma: M^+(K)\times X \rightarrow X$ is called a \textbf{general action} of $K$ on $X$ if the following two conditions hold true:

%\vspace{0.03in}

\begin{enumerate}
\item For each $\omega \in M^+(K)$ the map $\sigma_\omega:X \rightarrow X$ given by $\sigma_\omega(x) := \sigma(\omega, x)$ is continuous.
\item For any $\mu, \nu \in M^+(K)$, $x \in X$ we have that $\sigma(\mu*\nu,x) = \sigma(\mu, \sigma(\nu, x))$
\end{enumerate}
\end{definition}

We can define general actions of a left/ right/ semitopological semihypergroup $K$ on a Hausdorff space $X$ in the same manner as above. 

Let $\sigma: M^+(K)\times X \rightarrow X$ be a general action of $(K, *)$ on $X$.  Then $\sigma$ is called a seperately continuous general action, if for each $x\in X$ the map $\mu  \mapsto  \sigma(\mu, x): M^+(K)  \rightarrow  X$ is also continuous on $M^+(K)$ in the cone topology. Similarly, we say that $\sigma$ is a continuous general action, if $\sigma$ is jointly continuous on $M^+(K) \times X$. Finally, for any two subsets $N\subseteq M^+(K)$ and $V\subseteq X$, we define
$$ N.V := \{\sigma(\mu, x) : \mu \in N, x \in V\}.$$

%\end{remark}

%\vspace{0.2in}

We see that if a a left (right) topological semihypergroup $(K, *)$ has an identity $e$, then given any two distinct points in a Hausdorff space $X$, a certain general action $\sigma$ of $K$ on $X$ will always map each point away from the other. In particular, we have the following assertion.

\begin{proposition} \label{actlem2}
Let $(K, *)$ be a compact left topological semihypergroup, and $\sigma$ be a separately continuous general action of $K$ on a Hausdorff space $X$. Furthermore, let $K$ has an identity $e$ and $\sigma_{p_{_e}}$ is the identity map on $X$. Then for any two distinct points $x, y$ in $X$, there exist neighborhoods $N$ of $p_{_e}$, $U$ of $x$ and $V$ of $y$ such that $N.U \cap V = \O$.
\end{proposition}

\begin{proof}
Using Urysohn's Lemma, we first choose a continuous function $f: X \ra [-1, 1]$ such that $f(x) \neq f(y)=0$. Set $g:= f \circ \sigma$. 

%\vspace{0.2in}

Since $K$ is compact, the cone toplogy coincides with the weak$^*$-topology, and hence $M^+(K)$ is a  locally compact Hausdorff space \cite{DU} here. Also since $\sigma$ is separately continuous, the function $g: M^+(K) \times X \ra [-1, 1]$ is separately continuous. Hence we can find \cite{RU} a dense $G_\delta$ subset $M$ of $M^+(K)$ such that $g$ is continuous at each point $(\mu, x)$ in $M \times X$. Now set 
$$S:=\{\mu \in M^+(K): g(\mu, x) \neq g(\mu, y)\}.$$
Note that $S$ is non-empty since $\sigma_{p_e}$ is identity, and hence we have
\begin{eqnarray*}
g(p_e, x) &=& f(\sigma(p_e, x))\\
&=& f(x) \ \neq \ f(y) \ = \ f(\sigma(p_e, y)) \ = \ g(p_e, y).
\end{eqnarray*}

Also, $S$ is open in $M^+(K)$ since the map $\mu\mapsto g(\mu, x)-g(\mu, y): M^+(K)\ra \mathbb{R}$ is continuous. Hence $S \cap M \neq \O$. Pick any $\mu_0 \in S \cap M$, and in particular, set $\varepsilon_0:= |g(\mu_0, x) - g(\mu_0, y)|>0 $.

Since $g$ is jointly continuous at $(\mu_0, x)$, there exist open neighborhoods $N_0$ of $\mu_0$ and $U$ of $x$ such that
\begin{equation*} 
 |g(\nu, z) - g(\mu_0, x)|< \varepsilon_0/4 \mbox{ \ \ \ for any \ \ } \nu\in N_0, \ z\in U.
\end{equation*}
We further set $V:= \{z\in X : |g(\mu_0, z) - g(\mu_0, y)| < \varepsilon_0/4 \} $. Hence $x\notin V$, and since $g$ is separately continuous, $V$ is an open neighborhood of $y$.

But $K$ is left topological, and hence the map $L_{\mu_0} : M^+(K) \ra  M^+(K)$ is continuous. In particular, since $L_{\mu_0}(p_{_e}) = \mu_0$ and $N_0$ is an open neighborhood of $\mu_0$, we can find an open neighborhood $N$ of $p_{_e}$ in $M^+(K)$ such that $\{\mu_0\}*N = L_{\mu_0}(N) \subseteq N_0$. 
%\vspace{0.2in}

Now if possible, assume that $N.U \cap V \neq \O$. Pick some  $\nu \in N$, $u\in U$, $v\in V$ such that $\sigma(\nu, u)= v \in V$. Then we have
\begin{eqnarray*}
\varepsilon_0 = |g(\mu_0, x) - g(\mu_0, y)| &\leq& |g(\mu_0, x) - g(\mu_0*\nu, u)| + |g(\mu_0*\nu, u) - g(\mu_0, y)|\\
&<& \varepsilon_0/4 + |g(\mu_0, \sigma(\nu, u)) - g(\mu_0, y)|\\
&<& \varepsilon_0/4 + \varepsilon_0/4 < \varepsilon_0,
\end{eqnarray*}
where the third inequality follows from the choice of $N$ and $U$, since $\sigma(\nu, u)\in V$. Thus we arrive at a contradiction, and must have that $N.U \cap V = \O$.
\end{proof}

Next we proceed towards the main result of this section, that investigates the instances when separate continuity forces joint continuity of the general action of a left topological semihypergroup on a compact Hausdorff space. The proof of the result follows similar ideas as outlined in \cite{MI}, for the category of compact topological groups and semigroups. We show, with the help of \Cref{actlem2} and some important additional details, that for the general category of compact left topological semihypergroups, separate continuity of a general action forces joint continuity on a certain set of points. In addition, if the general action is linear in $M^+(K)$, then the result can be extended further to the whole measure space of the said set. 

We first prove the result for the case of all general actions $\sigma$ of $K$ on $X$, for which $\sigma_{p_{_e}}$ is the identity map on $X$, where $e$ is the identity of $K$.

%We state in order to generalize it for semihypergroups.

%\vspace{0.2in}

%\vspace{0.2in}

%First observe that since $K$ contains an identity element $e$, it is sufficient to prove the above result only for  We outline this fact in the following lemma.

%\vspace{0.2in}

\begin{theorem} \label{actmain1}
Let $(K, *)$ be a compact left topological semihypergroup  with identity $e$ and $\sigma$ be a separately continuous general action of $K$ on a compact Hausdorff space $X$ such that $\sigma(p_e, x)=x$ for each $x \in X$. Let $H \subseteq K$ be a hypergroup. Then $\sigma$ is continuous at each point $(p_{_g},x)$, for any $g \in Z(H)$, $x \in X$.
\end{theorem}

\begin{proof}

Pick and fix some $x \in X$ and let $W_x$ be an open neighborhood of $x=\sigma(p_e, x)$ in $X$. 

%%\vspace{0.2in}

Pick any $y \in X\setminus W_x$. Using  \Cref{actlem2} we can find open neighborhoods $N_y$ of $p_e$, $U_x^{^y}$ of $x$ and $V_y$ of $y$ such that $N_y.U_x^{^y} \cap V_y = \O$. We can find such a set of open neighborhoods for each $y \in X\setminus W_x$. Since $X\setminus W_x$ is compact, there exists $ y_1, y_2, \ldots, y_n \in X\setminus W_x$ such that 
$$X\setminus W_x \subseteq \Big{(} \cup_{i=1}^n V_{y_i} \Big{)} =: V.$$

%%\vspace{0.2in}

Set $N:=\cap_{i=1}^n N_{y_i} $ { and } $U:=\cap_{i=1}^n U_x^{^{y_i}}.$
Then clearly $N.U \cap V = \O$. Thus we get open neighborhoods $N$ of $p_e$, $U$ of $x$ such that $N.U \subseteq W_x$. Since this is true for any $x\in X$, we get that $\sigma$ is continuous at $(p_e, x)$ for each $x \in X$.

Next pick any $g\in Z(H)$, $x \in X$. Let $\{(\mu_\alpha, x_\alpha)\}$ be a net in $M^+(K)\times X$ that converges to $(p_g, x)$. Since $K$ is left topological, $L_{p_{g^-}}$ is continuous and hence $p_{g^-}*\mu_\alpha \ra  p_{g^-}*p_g =p_e$ in $M^+(K)$, as $g\in Z(H)$. But $\sigma$ is continuous at $(p_e, x)$, and hence 
$ \sigma(p_{g^-}*\mu_\alpha, x_\alpha) \ra \sigma(p_e, x)= x$ in $X$. Thus finally we have that
\begin{eqnarray*}
\sigma(\mu_\alpha, x_\alpha)&=&\sigma(p_e*\mu_\alpha, x_\alpha)\\
&=& \sigma\big{(}(p_g*p_{g^-})*\mu_\alpha, x_\alpha\big{)}\\
&=& \sigma(p_g, \sigma(p_{g^-}*\mu_\alpha,x_\alpha)) \longrightarrow \sigma(p_g, x),
\end{eqnarray*}
where the third equality follows since $(M(K), *)$ is associative, and the convergence follows from the continuity of the map $\sigma_{p_{_g}}:X \ra X$.
\end{proof}

Now we use \Cref{actmain1} to show that a similar conclusion holds true for any separately continuous general action $\sigma$ of $K$ on $X$.

\begin{theorem} \label{actmain2}
Let $(K, *)$ be a compact left topological semihypergroup  with identity $e$ and $\sigma$ be a separately continuous general action of $K$ on a compact Hausdorff space $X$. Let $H \subseteq K$ be a hypergroup. Then $\sigma$ is continuous at each point $(p_{_g},x)$, for any $g \in Z(H)$, $x \in X$.
\end{theorem}

\begin{proof}
For each $x\in X$, we define $\tilde{x} = \sigma(p_e, x)$, and set $L:= \{\tilde{x}: x\in X\} = \{p_e\}.X$. 

For any $\mu \in M^+(K)$, $x \in X$ we have that $M^+(K).L\subseteq L$, since
\begin{eqnarray*}
\sigma(\mu, \tilde{x}) \ = \ \sigma(\mu, \sigma(p_e, x)) &=& \sigma(\mu*p_e, x)\\
&=& \sigma(p_e*\mu, x) \ = \ \sigma(p_e, \sigma(\mu, x)) \in L.
\end{eqnarray*}
In particular, for each $x\in X$ we have that 
$$\sigma(p_e, \tilde{x}) = \sigma(p_e, \sigma(p_e, x))=\sigma(p_e*p_e, x)=\tilde{x}.$$
Hence the restricted general action $\tilde{\sigma} :=\sigma|_{M^+(K)\times L}: M^+(K)\times L \ra L$ is well defined and $\tilde{\sigma}_{p_{_e}}$ is the identity map on $L$. Note that $L$ is compact in $X$ since the map $\sigma_{p_{_e}}: X \ra X$ is continuous and $\sigma_{p_{_e}}(X)=L$. Moreover, since $\tilde{\sigma}$ is separately continuous when $L$ is equipped with the subspace topology induced from $X$, using \Cref{actmain1} we have that $\tilde{\sigma}$ is continuous at each point $(p_g,\tilde{x})$,  for any $g\in Z(H)$, $x\in X$.

%\vspace{0.2in}

Now pick any $g\in Z(H)$, $x\in X$ and let $\{(\mu_\alpha, x_\alpha)\}$ be any net in $M^+(K)\times X$ that converges to $(p_g, x)$. Since $\sigma_{p_{_e}}$ is continuous, the net $\{\tilde{x}_\alpha\}:=\{\sigma(p_e, x_\alpha)\}$ converges to $\tilde{x}$ in $X$, and hence in $L$. Hence $\{(\mu_\alpha, \tilde{x}_\alpha)\}$ converges to $(p_g, \tilde{x})$ in $M^+(K)\times L$. Thus finally, we have that
\begin{eqnarray*}
\sigma(\mu_\alpha, x_\alpha) &=& \sigma(\mu_\alpha * p_e, x_\alpha)\\
%&=& \sigma(\mu_\alpha, \sigma(p_e, x_\alpha))\\
&=& \tilde{\sigma}(\mu_\alpha, \tilde{x}_\alpha) \longrightarrow \tilde{\sigma}(p_g, \tilde{x})= \sigma(p_g*p_e, x) = \sigma(p_g, x).
\end{eqnarray*}
\end{proof}

\begin{remark}
Observe that if the general action $\sigma$ in \Cref{actmain2} is linear in $M^+(K)$ as well, then using the density of $M_F^+(K)$ in $M^+(K)$ we can conclude that separate continuity of $\sigma$ forces joint continuity of $\sigma$ on $M^+(Z(H))\times X$.
\end{remark}

Keeping this in mind, note that the action of a (semitopological) semihypergroup $K$ can also be defined in the following manner if the space $X$ on which $K$ acts, is a locally convex space. A locally convex Hausdorff topological vector space $E$ with  a family $Q$ of seminorms is simply denoted by $(E, Q)$ or $(E,\tau_Q)$. In rest of the article,  the topology assumed on $E$ (or on any Borel subset $X$ of $E$), is the (induced) topology $\tau_Q$ generated by $Q$.

\begin{definition}\label{action0}
Let $(K,*)$ be a (semitopological) semihypergroup and $X$ be a compact convex subset of $(E, Q)$. A map $\pi: K\times X \ra X$ is called a (semihypergroup) action if the following conditions are satisfied:
\begin{enumerate}
\item For any $s, t \in K$, $x\in X$ we have
$$\pi(s, \pi(t, x)) = \int_K \pi(\zeta,x) \ d(p_s*p_t)(\zeta).$$
\item Whenever $K$ has a two-sided identity $e$,  we have $\pi(e, x) =x$ for each $x\in X$.
\end{enumerate}
\end{definition}

Such an action $\pi$ of $K$ on $X$ is called (separately) continuous, if the map $\pi$ is (separately) continuous on $K\times X$. For each $s\in K$, we denote the map $x\mapsto \pi(s, x):X\ra X$ by $\pi_s$.

\begin{remark}
The above definition of a semihypergroup action is a particular case of  \Cref{action}. This is true since given a (separately) continuous action $\pi$ of $K$ on a compact convex subset $X$ of a locally convex Hausdorff vector space $(E, Q)$ in the sense of \Cref{action0}, one can simply define $$\sigma_\pi(\mu, x) := \int_K \pi(\zeta, x) \ d\mu(\zeta).$$
Since $\sigma_\pi(p_s, x):= \pi(s, x)$ for any $s\in K, x\in X$, and $M_F^+(K)$ is dense in $M^+(K)$ in the cone topology, we see that $\sigma_\pi$ is a (separately) continuous general action of $K$ on $X$  (\Cref{action}), since $\mu*\nu = \int_K\int_K \ (p_s*p_t) \ d\mu(s)\ d\nu(t)$ for each $\mu, \nu \in M^+(K)$ \cite{JE}. 

Note that the action $\sigma_\pi$ induced by $\pi$ is linear by construction. But a general action of a semihypergroup $K$ as defined in \Cref{action} need not be linear.
\end{remark}

In fact, it follows immediately that \Cref{action0} is structurally equivalent to the following definition of a semihypergroup action, induced by Jewett's definition \cite{JE} of hypergroup actions.

\begin{definition}\label{action1}
Let $K$ be a semihypergroup, and $(E, Q)$ be a separated locally convex space. Then a \textit{(semihypergroup) action} $\pi$ of $K$ on $E$ is a homomorphism from the associative algebra $M(K)$ to the algebra $L(E)$ of linear operators on $E$.

\smallskip

 In other words, a semihypergroup action $\pi$ of $K$ on $E$ is a bilinear map $(\mu, s) \mapsto \pi_\mu(s): M(K)\times E\ra E$ such that $\pi_{_{(\mu*\nu)}} = \pi_\mu \circ \pi_\nu$ on $E$ for each $\mu, \nu \in M(K)$.
\end{definition}

We say that a semihypergroup action $\pi: M(K)\times E\ra E$ is  \textit{(separately) $\tau$-continuous} if the map $\pi$ is (separately) continuous when $M(K)$ is equipped with a certain topology $\tau$ and $E$ is given the usual topology $\tau_Q$ induced by the associated family of  seminorms $Q$.

%%%%%%%%%%%%%%%%%%%%%%%%%%%%%%%%%%%%%%%%%%%%%%%%%%%%%%%%%%%%%%%%%%%%%%%%%%
%%%%%%%%%%%%%%%%%%%%%%%%%%%%%%%%%%%%%%%%%%%%%%%%%%%%%%%%%%%%%%%%%%%%%%%%%%%%%%%%

\section{Fixed-Point Properties and Amenability}
For the first part of this section, we use \Cref{action0} for semihypergroup actions, unless otherwise specified. Before proceeding to the main results of this section, we first briefly recall some definitions and results regarding amenability on function spaces of a semitopological semihypergroup. 

Let $K$ be any (semitopological) semihypergroup  and $\mathcal{F}$ a linear subspace of $C(K)$ containing constant functions. A function $m\in \mathcal{F}^*$ is called a \textit{mean} of $\mathcal{F}$ if we have that $||m|| = 1 = m(1)$, where $1$ denotes the constant function $\equiv 1$ on $K$. We denote the set of all means on $\mathcal{F}$ as $\mathcal{M}(\mathcal{F})$. If $\mathcal{F}$ is a left translation-invariant linear subspace of $C(K)$ containing constant functions, a mean $m$ of $\mathcal{F}$ is called a \textit{left invariant mean} (LIM) if $m(L_xf) = m(f)$ for any $x\in K$, $f\in \mathcal{F}$.

\begin{remark}\label{remaff}

Let $X$ be a convex subset of $(E, Q)$ and $Y$ be any locally convex space. Then a continuous map $T:X\ra Y$ is called \textit{affine} if for any $\alpha\in [0, 1]$, $x_1, x_2\in X$ we have that $$T(\alpha x_1 + (1-\alpha)x_2) = \alpha T(x_1) + (1-\alpha) T(x_2).$$
We say that a (separately) continuous action $\pi$ of $K$ on  $X$ is {affine} if for each $s\in K$, the map $\pi_s:X\ra X$ is affine. We denote by $A_f(X)$ the set of all affine maps $T: X\ra \mathbb{C}$. We know that $A_f(X)$ is a closed linear subspace of $C(X)$, \textit{i.e,} is itself a convex space in $C(K)$ \cite{PA}. 

For each $x\in X$, we define the evaluation map $\varepsilon_x:C(X) \ra \mathbb{C}$ as $\varepsilon_x(f) :=f(x)$ for each $f\in C(X)$. Since the constant function $1$ is in $A_f(X)$, we have that each $\varepsilon_x$ is a mean on $A_f(X)$. In fact, it can be shown \cite{PA} that the map $x\mapsto \varepsilon_x|_{A_f(X)} : X \ra \mathcal{M}(A_f(X))$ is an affine homeomorphism. The set of all evaluation maps $\{\varepsilon_x:x\in K\}$ is denoted by $\varepsilon(K)$. 

Similarly, for each $\mu\in M(K)$, we define the evaluation map $\varepsilon_\mu$ at the measure $\mu$ as the functional $f\mapsto \int_K f \ d\mu:C(K) \ra \mathbb{C}$. Thus for each $x\in K$, we simply have that $\varepsilon_{p_{_x}} = \varepsilon_x$ on $C(K)$.
\end{remark}

In the above setting, let $\pi$ be an affine action of $K$ on a compact convex subset $X\subseteq E$. Then $\pi$ naturally induces an affine action of $K$ on $A_f(X)$ when $K$ is commutative. We provide a brief account of this fact below.

\begin{proposition}\label{inact}
Let $K$ be a commutative semitopological semihypergroup, $X$ be a compact convex subset of a locally convex Hausdorff space $(E, Q)$ and $\pi$ be an affine action of $K$ on $X$. Then $\pi$ naturally induces an affine action $\tilde{\pi}$ of $K$ on $A_f(X)$.
\end{proposition}
\begin{proof}
We define $\tilde{\pi}: K \times A_f(X) \ra A_f(X)$ as
$$\tilde{\pi}(s, f)(x) := f(\pi(s, x)),$$
for each $s\in K, f\in A_f(X), x\in X$. First note that $\tilde{\pi}$ is well-defined since both $\pi_s$ and  $f:X\ra \mathbb{C}$ are affine for each $s\in K$. Moreover, if $K$ has an identity $e$, then for each $f\in A_f(X), x\in X$, we have $\tilde{\pi}_e(f)(x) = f(\pi_e(x))=f(x)$. Now for any $s, t \in K$, $f\in A_f(X)$, $x\in X$ we have:
\begin{eqnarray*}
\tilde{\pi}(s, \tilde{\pi}(t, f))(x) \ = \ \tilde{\pi}_s(\tilde{\pi}_t(f))(x) &=& \tilde{\pi}_t (f)(\pi_s(x))\\
&=& f(\pi_t(\pi_s(x)))\\
&=& f \Big{(} \int_K \pi_\zeta(x) \ d(p_t*p_s)(\zeta)\Big{)}\\
&=& \int_K f(\pi_\zeta(x)) \ d(p_t*p_s)(\zeta)\\
&=& \int_K \tilde{\pi}_\zeta (f)(x) \ d(p_t*p_s)(\zeta) \ = \ \int_K \tilde{\pi}(\zeta,  f)(x) \ d(p_s*p_t)(\zeta).\\
\end{eqnarray*}
where the third last equality follows since $f$ is affine on $X$. Also, since $A_f(X) \subset C(X)$, we have that $\tilde{\pi}$ is (separately) continuous if $\pi$ is so.
\end{proof}

Now we proceed to derive a sufficient condition for a semihypergroup action to attain a fixed point. We see that a (separately) continuous action of a commutative   semihypergroup $K$ on a compact convex subset $X$ of $(E,Q)$ will always attain a fixed point. This is a direct consequence of the fact we showed previously in \cite{CB1}, that any commutative semihypergroup $K$ is left amenable. Hence in particular, $K$ will admit a left-invariant mean on the spaces of (weakly) right uniformly continuous functions and almost periodic functions on $K$. The proof  is inspired by techniques used in Rickert's Fixed Point Theorem \cite{RI}, proved for topological groups.

\begin{theorem} \label{main2}
Let $K$ be a commutative (semitopological) semihypergroup and $X$ be a compact convex subset of a locally convex Hausdorff vector space $(E, Q)$. Then any separately continuous affine action of $K$ on $X$ has a fixed point.
\end{theorem}

\begin{proof}
Let $\pi$ be a separately continuous affine action of $K$ on $X$. Recall from \Cref{inact} that $\pi$ naturally induces an affine action $\tilde{\pi}$ of $K$ on $A_f(X)$ where for each $s\in K, f\in A_f(X), x\in X$ we have $$\tilde{\pi}_s (f)(x) = f(\pi_s(x)).$$
Now for each $x\in X$, we define a function $\phi_x:A_f(X)\ra C(K)$ as
$$\phi_x (f)(s) := f(\pi_s(x)),$$
for each $f\in A_f(X), s\in K$. Since $\pi$ is separately continuous on $K\times X$, for any fixed $x\in X$ and a net $(s_\alpha)$ in $K$ such that $s_\alpha \ra s$ for some $s\in K$, we have that $\pi_{s_\alpha}(x)\ra \pi_s(x)$ in $X$. Since $f$ is continuous on $X$, we have that  $\phi_x(f)(s_\alpha) = f(\pi_{s_\alpha}(x)) \ra f(\pi_s(x)) = \phi_x(f)(s)$, \textit{i.e,} $\phi_x(f)$ is indeed continuous on $K$, which is also bounded as $f\in C(X)$.

For each $s\in K, f\in A_f(K), x \in X$ we further have that:
\begin{eqnarray*}
L_s\phi_x(f) (t) &=& \phi_x(f)(s*t)\\
&=& \int_K \phi_x(f)(\zeta) \ d(p_s*p_t)(\zeta)\\
&=& \int_K f(\pi_\zeta(x)) \ d(p_s*p_t)(\zeta)\\
&=& f \Big{(}\int_K \pi_\zeta(x) \ d(p_s*p_t)(\zeta) \Big{)}\\
&=& f(\pi_s(\pi_t(x))) \ = \ \tilde{\pi}_s(f)(\pi_t(x)) \ = \ \phi_x(\tilde{\pi}_s(f))(t),
\end{eqnarray*}
for each $t\in K$, where the fourth equality holds true since $f$ is affine. Hence we have that 
\begin{equation}\label{orbiteq}
L_s\phi_x(f) = \phi_x(\tilde{\pi}_s(f))
\end{equation}
for each $s\in K$, \textit{i.e,} the translation-orbit of the function $\phi_x(f)\in C(K)$ coincides with the image of the $\tilde{\pi}$-orbit of $f$, under the map $\phi_x$.

Since $K$ is commutative, we know \cite{CB1} that $K$ is amenable, \textit{i.e,} there exists a translation invariant mean $m$ on $C(K)$. Fix  any $x_0\in X$, and consider the map $m \circ \phi_{x_0} : A_f(X)\ra \mathbb{C}$. Note that $m \circ \phi_{x_0}$ is also a mean on $A_f(X)$ since $\phi_{x_0}$ is a  bounded linear map and $A_f(X)\subset C(X)$. Hence it follows from \Cref{remaff} that there exists some $z_0\in X$ such that $m \circ \phi_{x_0}= \varepsilon_{z_0}$.

In particular,  since $\phi_{x_0}(f) \in C(K)$ for any $f\in A_f(X)$, we have that $m(L_s \phi_{x_0}(f)) = m(\phi_{x_0}(f))$ for each $s\in K$. Hence using the interplay between orbits derived in \Cref{orbiteq} we have that 
\begin{eqnarray*}
  \ \ \ \ \  m(\phi_{x_0}(\tilde{\pi}_s(f))) &=& m(\phi_{x_0}(f)),\\
 \Rightarrow \ (m\circ \phi_{x_0}) (\tilde{\pi}_s(f)) &=& (m\circ \phi_{x_0}) (f),\\
  \Rightarrow \hspace{0.5in} \varepsilon_{z_0} (\tilde{\pi}_s(f)) &=& \varepsilon_{z_0} (f),\\
 \Rightarrow \ \hspace{0.5in}  f(\pi_s(z_0)) &=& f(z_0),
\end{eqnarray*}
for each $f\in A_f(X)$, $s\in K$. Hence in particular, we must have that $\pi(s, z_0)=z_0$ for each $s\in K$, as required.
\end{proof}

On the other hand, recall \cite{CB1} that the space of almost periodic functions $AP(K)$ is a translation-invariant, norm-closed linear subspace of $C(K)$ such that $AP(K)\subseteq UC(K)$. In the following result, we see that the fixed point property outlined in the statement of \Cref{main2} is a sufficient condition for any semihypergroup $K$ to admit a left invariant mean on $AP(K)$. The idea of the proof is similar to the standard techniques used in \cite{MIT} for the case of locally compact semigroups, although the structures of the function spaces used are substantially different. 

\begin{theorem}\label{main3}
Let $K$ be a (semitopological) semihypergroup such that any jointly continuous affine action $\pi$ of $K$ on a compact convex subset $X$ of a locally convex vector space $(E, Q)$ has a fixed point. Then there exists a LIM on $AP(K)$.
\end{theorem}

\begin{proof}
In particular, set $E:= AP(K)^*$ and equip it with the weak$^*$-topology. Set $X:= \mathcal{M}(AP(K))$. Since $X\subset \mathcal{B}_1(AP(K)^*)$, we have that $X$ is compact in the induced weak$^*$-topology. Now consider the map $\pi: K\times X \ra X$ given by $$\pi(s, m)(f) = \pi_s(m) (f) := m(L_sf)$$
for each $s\in K, m\in X, f\in AP(K)$. Note that for any $s, t \in K$, we have
\begin{eqnarray*}
\int_K L_\zeta f(y) \ d(p_s*p_t)(\zeta) &=& \int_K R_yf(\zeta) \ d(p_s*p_t)(\zeta)\\
&=& \int_K f \ d(p_s*p_t*p_y)\\
&=& f(s*t*y)\\
&=& L_sf(t*y) \ = \ (L_t\circ L_s)(y),
\end{eqnarray*} 
for each $y\in K$. Hence $\pi$ is indeed an action of $K$ on $X$ since for each $s, t\in K, m\in X$, $f\in AP(K)$ we have that
\begin{eqnarray*}
\pi(s, \pi(t, m))(f) &=& \pi_t(m)(L_s f)\\
&=& m(L_t \circ L_s)(f)\\
&=& m\Big{(} \int_K L_\zeta f \ d(p_s*p_t)(\zeta) \Big{)}\\
&=&  \int_K m (L_\zeta f) \ d(p_s*p_t)(\zeta) \ = \ \int_K \pi(\zeta, m)(f) \ d(p_s*p_t)(\zeta),
\end{eqnarray*}
where the fourth equality follows since $m \in AP(K)^*$. Now pick any $s\in K, m\in X$. Let $(s_\alpha)$ and $(m_\beta)$ be sequences in $K$ and $X$ respectively, such that $s_\alpha \ra s$ and $m_\alpha \ra m$ in $K$ and $X$ respectively. For each $f\in AP(K) \subseteq UC(K)$ we have that $||(L_{s_\alpha}f - L_sf)||_u \ra 0$. Hence for each $f\in AP(K)$ we have that
\begin{eqnarray*}
\lim_{\alpha, \beta} |\pi(s_\alpha, m_\alpha)(f) - \pi(s, m)| &=& \lim_{\alpha, \beta} |m_\beta(L_{s_\alpha}f) - m(L_sf)|\\
&\leq & \lim_{\alpha, \beta} |m_\beta(L_{s_\alpha}f - L_sf)| + \lim_{\alpha, \beta} |(m_\beta - m)(L_sf)|\\
&\leq & \lim_{\alpha, \beta} ||m_\beta|| \: || (L_{s_\alpha}f - L_sf)||_u + \lim_\beta ||(m_\beta - m)|| \: ||L_sf||_u\\
&=& \lim_\alpha || (L_{s_\alpha}f - L_sf)|| + \lim_\beta ||(m_\beta - m)|| \: ||f||_u \ \longrightarrow \ 0,
\end{eqnarray*}
since $f$ is bounded on $K$. Hence we have that the action $\pi$ is jointly continuous on $K\times X$. Since $\pi$ is also affine by construction, we have that $\pi$ admits a fixed point, \textit{i.e,} there exists some $m_0\in X$ such that $\pi(s, m_0) = m_0$ for each $s\in K$. Thus we have that $$m_0(f) = \pi(s, m_0)(f) = m_0(L_sf),$$ for each $s\in K, f\in AP(K)$, \textit{i.e, } $m_0$ is a left invariant mean on $AP(K)$.
\end{proof}

Next, in the final part of the section, we consider a different fixed-point property involving measure algebra-homomorphisms, that uses \Cref{action1} for semihypergroup actions. In particular, we consider the topology $\tau_{_{AP(K)}}:=\sigma(M(K), AP(K))$ on $M(K)$ induced by $AP(K)$, and investigate the fixed points of $\tau_{_{AP(K)}}$-continuous actions of $K$ on a separated locally convex space $(E, Q)$, with respect to the space of probability measures $P(K)$. Before we proceed further along this line of investigation, we first discuss a couple of lemmas regarding the density properties of evaluation maps $\varepsilon_x$, $x\in K$, in the closed unit ball of a norm-closed subspace of $C(K)$,  which are imperative to the proof of the next main results.  

Recall \cite{BN} that a  subset $E$ of a locally convex space $V$ is called \textit{balanced} or \textit{circled} if for any $x\in E$, we have that  $\alpha x\in E$ for each $\alpha \in \mathbb{F}$ such that $|\alpha|\leq 1$, where $\mathbb{F}= \mathbb{C} \ (\mbox{or } \mathbb{R})$ is the associated field of scalars. Hence the\textit{ balanced convex hull} or \textit{circled convex hull} of $E$ is the smallest balanced, convex set in $V$ that contains $E$, and is denoted as $cco(E)$. It can be shown\cite{BN} easily that
	$$cco(E) = \Big{\{}\sum_{i=1}^n \alpha_ix_i : x_i\in E, \alpha_i\in \mathbb{F} \mbox{ for } 1\leq i \leq n; \ n \in \mathbb{N} \mbox{ and } \sum_{i=1}^n |\alpha_i| \leq 1 \Big{\}}.$$
	Note that although $cco(E)$ equals the convex hull of the smallest balanced set in $V$ containing $E$, it may strictly contain the smallest balanced set in $V$ containing the convex hull of $E$ \cite{BN}.

	The following couple of results hold true even when the semihypergroup $K$ is substituted by any locally compact Hausdorff space, in general. We include brief proofs here for convenience of readers.
	
	\begin{lemma}\label{genlem}
		Let $\F$ be a norm-closed, conjugate-closed linear subspace of $C(K)$. Then $cco(\varepsilon(K))$ is weak$^*$-dense in $\mathcal{B}_1(\F^*)$.
	\end{lemma}

\begin{proof}
Let $A:= \overline{cco}^{w^*}(\varepsilon(K))$. Now if possible, let there exists some $\phi\in \mathcal{B}_1(\F^*)\setminus A$. Let $V$ denote the locally convex topological vector space $\F^*$ equipped with the  weak$^*$-topology.  Since $A \subset \mathcal{B}_1(\F^*)$, $A$ is a compact convex subset of $V$. Hence using the strong separation theorem for locally convex spaces \cite[Theorem 7.8.6]{BN}, we can find a weak$^*$-continuous linear functional $\omega$ on $\F^*$ such that $$\alpha:=\sup\{Re(\omega(\psi)): \psi\in A\} < Re(\omega(\phi)).$$
	Since $0\in A$ and $A$ is balanced, we know that $Re(\omega(\phi))>\alpha>0$. Now, we define a weak$^*$-continuous map $\rho:\F^*\ra \mathbb{C}$ as $$\rho(\psi):= \frac{1}{\alpha} Re(\omega(\psi)) - \frac{i}{\alpha} Re(\omega(i\psi)),  $$
	for each $\psi \in \F^*$. Note that $\rho$ is a linear functional since for any $\psi\in \F^*$, we have 
$$\rho(i\psi)= \frac{i}{\alpha} Re(\omega(\psi)) + \frac{1}{\alpha} Re(\omega(i\psi)) = i\rho(\psi).$$
Since $(V,\F)$ is a dual pair, there exists \cite{BN} an unique $f_0\in \F$ such that $$\rho(\psi)= \langle \psi, f_0\rangle = \psi(f_0)$$
for each $\psi \in \F^*$. Hence  in particular, we have that
\begin{eqnarray*}
|\phi(f_0)| \ = \ |\langle \phi, f_0 \rangle| &=& |\rho(\phi)|\\
&\geq& \frac{1}{\alpha} Re(\omega(\phi)) \ > \ 1.
\end{eqnarray*}	
But this poses a contradiction to the fact that $||\phi||<1$, since we have that $||f_0||\leq 1$ as well, which can be verified as the following:
\begin{eqnarray*}
	Re(\langle \psi, f_0\rangle) &=& Re (\rho(\psi))\\
	&=& \frac{1}{\alpha} Re(\omega(\psi)) \ \leq 1,
		\end{eqnarray*}
for each $\psi \in A$. Since $A$ is balanced, we have that $Re(\langle e^{i\theta}\psi, f_0\rangle) \leq 1$ for each $\theta \in \mathbb{R}$, and hence $|\langle \psi, f_0\rangle|\leq 1$ for each $\psi\in A$. Thus in particular, for each $x\in K$ we have that $$|f_0(x)|=|\langle \varepsilon_x, f_0\rangle| \leq 1.$$ 
\end{proof}

\begin{lemma} \label{genlem0}
Let $\F$ be a norm-closed, conjugate-closed linear subspace of $C(K)$. Then $\{\ep_\mu: \mu\in P(K)\}$ is weak$^*$-dense in the set $\mathcal{M}(\F)$ of all means on $\F$.
\end{lemma}

\begin{proof}
Pick any $m\in \mathcal{M}(\F)\subseteq \mathcal{B}_1(\F^*)$. Using \Cref{genlem} and \Cref{remaff}, we get a net $\{\mu_\alpha\}$ in $M(K)$ such that  $\varepsilon_{\mu_\alpha}\overset{w^*}{\longrightarrow}  m$ on $\F^*$. But we have that $m(f)\geq 0$ for any $f\geq 0$ in $C(K)$ and $m(1)=1$, since $m\in \mathcal{M}(\F)$. Hence we must have that $\mu_\alpha \in M^+(K)$ eventually, and hence $||\mu_\alpha||=\mu_\alpha(K)$ must converge to $1$. 

Hence without loss of generality, we may assume that each $\mu_\alpha \in M^+(K)$. Setting $\nu_\alpha:=\mu_\alpha/||\mu_\alpha||$ we get that $\ep_{{\nu_\alpha}}\overset{w^*}{\longrightarrow} m$, where $\nu_\alpha \in P(K)$ for each $\alpha$.
\end{proof}

Next, we recall some properties of the dual space $AP(K)^*$, in order to use its algebraic structure to provide certain necessary and sufficient conditions for the existence of a left-invariant mean on the space of almost periodic functions.

Let $K$ be a semitopological semihypergroup, and $\mathcal{F}$ be a left (resp. right) translation-invariant linear subspace of $C(K)$. For each $\omega \in \mathcal{F}^*$, the left (resp. right) introversion operator $T_\omega$ (resp. $U_\omega$) determined by $\omega$ is the map $T_\omega : \mathcal{F} \ra B(K)$ (resp. $U_\omega : \mathcal{F} \ra B(K)$) defined as $ T_\omega f(x) := \omega(L_x f)$ (resp. $U_\omega f(x) := \omega(R_x f)$) for each $f\in \F, x\in K$. We say that a left (resp. right) translation-invariant linear subspace $\mathcal{F}$ is left-introverted (resp. right-introverted) if $T_\omega f \in \mathcal{F}$ (resp. $U_\omega f \in \mathcal{F}$) for each $\omega \in \mathcal{F}^*$, $f\in \mathcal{F}$. A translation-invariant subspace $\F$ is called introverted if it is both left and right introverted. 

In a previous work \cite{CB1} we discussed the basic properties of introversion operators for a semitopological semihypergroup, and in turn showed that the function-space $AP(K)$ is translation-invariant and introverted. Hence although the topological space $K$ lacks an algebraic structure on itself, we can define both left and right Arens product \cite{AR} on the dual space $AP(K)^*$. It can be seen easily\cite{CB1} that the left Arens product $\lozenge$ on $AP(K)^*$ coincides with the definition $$(m \lozenge n)(f) := m(T_n(f)),$$
for each $m, n \in AP(K)^*$, $f\in AP(K)$. Similarly, for each $m, n \in AP(K)^*$ the right Arens product $\square$ on $AP(K)^*$ coincides with the definition $$(m \square n)(f) := n(U_m(f)),$$ 
for each $f\in AP(K)$. In fact, it can be shown \cite[Theorem 4.27]{CB1} that $AP(K)^*$ is Arens-regular, i.e, $(m \lozenge n)= (m \square n)$ for any $m, n \in AP(K)^*$. Hence from here onwards, without any loss of generality, we consider the algebra $(AP(K)^*, \star)$ where $(m\star n)(f):= m(T_n(f))$ for each $f\in AP(K)$.

\begin{remark}\label{remw}
For any $\mu, \nu \in M(K)$, we have that $\varepsilon_\mu \star \varepsilon_\nu = \varepsilon_{{\mu * \nu}}$ on $AP(K)$,  since for any $f\in AP(K)$, we have the following:
\begin{eqnarray*}
	(\ep_\mu \star \ep_\nu) (f) &=& \ep_\mu(T_{\ep_\nu}(f)) \\
	&=& \int_K T_{\ep_\nu}(f)(x) \ d\mu(x)\\
	&=& \int_K \ep_\nu(L_xf) \ d\mu(x)\\
	&=& \int_K \int_K L_xf (y) \ d\nu(y) \ d\mu(x)\\
	&=& \int_K \int_K f(x*y) \ d\mu(x) d\nu(y) \ = \int_K f \ d(\mu*\nu) \ = \ \ep_{\mu*\nu}(f),
	\end{eqnarray*}
where the fifth equality holds true since $f$ is bounded on $K$, and both $||\mu||, ||\nu||$ are finite.

\end{remark}

The following results are inspired by the techniques and results proved in \cite{WO}, for locally compact semigroups. Note that unlike locally compact groups and semigroups, here the spaces $LUC(K), RUC(K), UC(K)$ and $AP(K)$ do not necessarily form  sub-algebras of $C(K)$, with respect to the point-wise product.

\begin{theorem}\label{thmequiv}
	Let $K$ be a semitopological semihypergroup. Then the following statements are equivalent.
	\begin{enumerate}
	\item $AP(K)$ admits a LIM.\\
	\item $\exists$ a mean $ m $ on $AP(K)$ such that $\ep_\mu\star m=m$ for each $\mu\in P(K)$.\\
	\item $\exists$ a net $\{\mu_\alpha\}$ in $P(K)$ such that for each $\mu \in P(K)$, the net $\{(\mu * \mu_\alpha - \mu_\alpha)\}$ converges to $0$ in $(M(K), \sigma(M(K), AP(K)))$.
\end{enumerate} 

\smallskip
Moreover, if $m$ is a LIM on $AP(K)$, then  $\ep_\mu\star m=m$  for each $\mu\in P(K)$.
\end{theorem}
\begin{proof}
 $(1)\Rightarrow (2)$: Let $m$ be a LIM on $AP(K)$, and $\mu\in P(K)$. Since $||\ep_\mu||=1= \ep_\mu(1)$, it immediately follows from \Cref{genlem0} that there exists a net $\{\phi_\alpha\}$ in $\{\ep_\nu : \nu \in P(K)\}$ such that $\phi_\alpha\overset{w^*}{\longrightarrow} \ep_\mu$  on $AP(K)^*$. But $\{\ep_\nu : \nu \in P(K)\}=co(\ep(K))$, and hence for each $\alpha$ there exists some $n_\alpha\in \mathbb{N}$ such that $\phi_\alpha = \sum_{i=1}^{n_\alpha} \lambda_i^\alpha \ep_{x_i^\alpha}$, where each $x_i^\alpha\in K$, $0\leq \lambda_i^\alpha\leq 1$ for $1\leq i \leq n_\alpha$ such that $\sum_{i=1}^{n_\alpha} \lambda_i^\alpha  = 1$.  Hence $\sum_{i=1}^{n_\alpha} \lambda_i^\alpha p_{x_i^\alpha}$ converges to $\mu$ in $(M(K), \tau_{_{AP(K)}})$, as for each $f\in AP(K)$,  we have that
 \begin{eqnarray*}
 	\int_K f \ d\Big{(} \sum_{i=1}^{n_\alpha} \lambda_i^\alpha p_{x_i^\alpha}\Big{)} &=& \sum_{i=1}^{n_\alpha} \lambda_i^\alpha f({x_i^\alpha}) \\ 
 	&=& \sum_{i=1}^{n_\alpha} \lambda_i^\alpha \ep_({x_i^\alpha}) (f) \ = \phi_\alpha(f) \ \ra \  \int_K f \ d\mu.\\
 	\end{eqnarray*}
Now, let $F$ be a norm-bounded subset of $AP(K)^*$. Consider the map $\Omega_F: M(K)\times F \ra AP(K)^*$ given by $$\Omega_F (\mu, \omega) := \ep_\mu \star \omega ,$$ 
for each $(\mu, \omega) \in M(K)\times F$. Pick $M>0$ such that $||\rho||\leq M$ for each $\rho\in F$. 

let $\mu_\alpha \longrightarrow \mu$ in $(M(K), \sigma(M(K), AP(K)))$ and $\omega\in AP(K)^*$. Then for each $f\in AP(K)$ we have that
\begin{eqnarray*}
\Omega_F(\mu_\alpha, \omega)(f) &=& (\ep_{\mu_\alpha} \star \omega)(f) \\
&=& \int_K T_\omega f (x) \ d\mu_\alpha(x) \\
&\longrightarrow & \int_K T_\omega f (x) \ d\mu(x) \ = \ \ep_\mu(T_\omega f) \ = \ \Omega_F(\mu, \omega),
\end{eqnarray*}
since $AP(K)$ is left-introverted \cite{CB1}, and hence $T_\omega f \in AP(K)$, as $f\in AP(K)$.

On the other hand, let $\{\omega_\alpha\}$ be any net in $AP(K)^*$ such that $\omega_\alpha \overset{w^*}{\longrightarrow} \omega$ for some $\omega\in AP(K)^*$, and consider the family of functions $\{T_{\omega_\alpha}f\}$ for some $f\in AP(K)$. For any $x, y \in K$ we have
\begin{eqnarray*}
	|T_{\omega_\alpha} f(x) - T_{\omega_\alpha} f(y)| &=& |{\omega_\alpha} (L_xf) - {\omega_\alpha}(L_yf)|\\
	&=& |{\omega_\alpha}(L_xf - L_yf)|\\
	&\leq & ||{\omega_\alpha}||\:  ||L_xf-L_yf||_u \ \leq \ M ||L_xf - L_yf||_u  \ra 0 \mbox{ whenever } x\ra y,
	\end{eqnarray*}
since $AP(K)\subset UC(K)$ \cite{CB1} and hence the map $x\mapsto L_xf: K \ra (C(K), ||\cdot||_u)$ is  continuous for each $f\in AP(K)$.  Hence $\{T_{\omega_\alpha}f\}$ is an equicontinuous family of functions in $(C(K), ||\cdot||_u)$. Therefore the topology of pointwise convergence will coincide with the topology of uniform convergence on compact sets. For each $x\in K$, we have that $$T_{\omega_\alpha}f(x) = \omega_\alpha(L_xf)\ra \omega(L_xf) = T_\omega f(x),$$ since $L_xf\in AP(K)$ and $\omega_\alpha \overset{w^*}{\longrightarrow} \omega$. Hence $T_{\omega_\alpha}f \ra 
  T_\omega f$ uniformly on compact subsets of $K$. Thus for any $\mu\in M_C(K)$ and $f\in AP(K)$ we have: 
 \begin{eqnarray*}
 \Omega_F(\mu, \omega_\alpha)(f) &=& (\ep_\mu \star {\omega_\alpha})(f)\\
 &=& \ep_\mu(T_{\omega_\alpha}f)\\
 &=& \int_K T_{\omega_\alpha}f(x) \ d\mu(x)\\
 &\longrightarrow & \int_K T_\omega f (x) \ d\mu(x) \ = \ (\ep_\mu \star \omega)(f) \ = \ \Omega_F(\mu, \omega)(f),
 \end{eqnarray*}
where the fourth implication holds true since $supp(\mu)$ is compact.  But note that $M_C(K)$ is norm-dense in $M(K)$ and for each $f\in AP(K)$ and any $\alpha$ we have that
$$ ||T_{\omega_\alpha}(f)||_u = \underset{x\in K} {\sup} \ |\omega_\alpha(L_xf)|  \leq \underset{x\in K} {\sup} \ (||\omega_\alpha|| \: ||L_xf||) \leq M ||f||.$$
Hence for each $\mu\in M(K)$ we have that $\Omega_F(\mu, \omega_\alpha) \overset{w^*}{\longrightarrow} \Omega_F(\mu, \omega)$. Thus we see that $\Omega_F$ is separately $\big{(} \tau_{_{AP(K)}} \times \mbox{ weak}^*\mbox{-topology}\big{)}- \mbox{ weak}^*\mbox{-topology}$ continuous.

In particular taking $F=P(K)$ and $\omega = m$,  since $\sum_{i=1}^{n_\alpha} \lambda_i^\alpha p_{x_i^\alpha}$ converges to $\mu$ in $(M(K), \tau_{_{AP(K)}})$, we have that 
$$(\phi_\alpha \star m) = \Big{(}\ep_{_{\sum_{i=1}^{n_\alpha} \lambda_i^\alpha p_{x_i^\alpha} }} \star m \Big{)} \overset{w^*}{\longrightarrow} (\ep_\mu \star m) \mbox{ on } AP(K)^*.$$
But for each $\alpha$, we also have that
\begin{eqnarray*}
(\phi_\alpha \star m)(f) &=& \phi_\alpha (T_m f)\\
&=& \sum_{i=1}^{n_\alpha} \lambda_i^\alpha (T_mf)({x_i^\alpha})\\
&=& \sum_{i=1}^{n_\alpha} \lambda_i^\alpha m(L_{x_i^\alpha}f) \ = \ \sum_{i=1}^{n_\alpha} \lambda_i^\alpha m(f)  \ = \ m(f),
\end{eqnarray*}
for each $f\in AP(K)$. Since the weak$^*$-topology is Hausdorff, thus we must have that $\ep_\mu \star m = m$ on $AP(K)$.

$(2)\Rightarrow(3)$: Let $m$ be a mean on $AP(K)$ such that $\ep_\mu \star m = m$ for each $\mu\in P(K)$. It follows from \Cref{genlem0}  that there exists a net $\{\mu_\alpha\}$ in $P(K)$ such that $\ep_{\mu_\alpha} \overset{w^*}{\longrightarrow} m$ in $AP(K)^*$. Hence in particular, taking $F=P(K)$, we have that $$\ep_\mu \star \ep_{\mu_\alpha} = \Omega_F(\mu, \ep_{\mu_\alpha}) \overset{w^*}{\longrightarrow} \Omega_F(\mu, m) = \ep_\mu \star m = m.$$
Hence for each $\mu\in P(K)$ and $f\in AP(K)$ we have the following.
\begin{eqnarray*}
\int_K f \ d(\mu*\mu_\alpha -\mu_\alpha) &=& \int_K f \ d(\mu*\mu_\alpha) - \int_K f \ d\mu_\alpha\\
&=& \ep_{(\mu*\mu_\alpha)}(f) - \ep_{\mu_\alpha}(f)\\
&=& (\ep_\mu \star \ep_{\mu_\alpha})(f) - \ep_{\mu_\alpha}(f)\\
&\longrightarrow & m -m \ = \ 0,
\end{eqnarray*}
where the third equality follows from \Cref{remw}.

$(3)\Rightarrow (2)$: Let $\{\mu_\alpha\}_{\alpha\in I}$ be a net in $P(K)$ such that for each $\mu \in P(K)$, the net $\{(\mu * \mu_\alpha - \mu_\alpha)\}_{\alpha\in I}$ converges to $0$ in $\tau_{_{AP(K)}}$. Since $\mathcal{M}(AP(K))\subset \mathcal{B}_1(AP(K)^*)$ is weak$^*$-compact, it follows from \Cref{genlem0} that  there exists a weak$^*$-convergent subnet $\{\mu_{\beta}\}_{\beta\in I'}$ of $\{\mu_\alpha\}_{\alpha\in I}$, $I'\subseteq I$, such that $\ep_{\mu_{_\beta}} \overset {w^*}{\longrightarrow } m$ for some $m\in \mathcal{M}(AP(K))$.

Now for each $\mu \in P(K)$, using \Cref{remw} and the continuity of $\Omega_F$ with $F= {P(K)}$, we get that $$\ep_{_{(\mu * \mu_{_\beta})}} = (\ep_\mu \star \ep_{\mu_{_\beta}}) = \Omega_F(\mu, \ep_{\mu_{_\beta}}) \overset {w^*}{\longrightarrow } \Omega_F(\mu, m)= (\ep_\mu \star m).$$
Thus finally, for any $f\in AP(K)$, $\mu\in P(K)$, using the given criterion we have the following, as required.
\begin{eqnarray*}
\big{(}(\ep_\mu \star m) - m \big{)}(f) &=& \lim_\beta{}^{w^*} \ep_{_{(\mu * \mu_{_\beta})}}(f) - \lim_\beta{}^{w^*} \ep_{\mu_{_\beta}}(f)\\
&=& \lim_\beta{}^{w^*} \int_K f \ d(\mu * \mu_{_\beta}) - \lim_\beta{}^{w^*} \int_K f \ d\mu_{_\beta} \\
&=& \lim_\beta{}^{w^*} \int_K f \ d(\mu * \mu_{_\beta} - {\mu_{_\beta}}) \ = \ 0.
\end{eqnarray*}
$(2)\Rightarrow (1)$: Let $m$ be a mean on $AP(K)$ such that $\ep_\mu\star m=m$ for each $\mu\in P(K)$. In particular, pick any $x_0\in K$ and consider $\mu= p_{_{x_0}}$. Then it follows immediately that $m$ is a LIM on $AP(K)$ since for each $f\in AP(K)$ the following holds true:
\begin{eqnarray*}
m(L_{x_0} f) &=& T_m f (x_0)\\
&=& \ep_{p_{_{x_0}}} (T_m f) \ = \ (\ep_{p_{_{x_0}}}\star m)(f) \ = \ \mu(f).
\end{eqnarray*}
\end{proof}

%\begin{theorem}
%$AP(K)$ has a LIM if and only if there exists a net $\{\mu_\alpha\}$ in $P(K)$ such that for each $\mu \in P(K)$, the net $\{(\mu * \mu_\alpha - \mu_\alpha)\}$ converges to $0$ in $\sigma(M(K), AP(K))$.
%\end{theorem}

We now proceed to the final result of this section.   For convenience, we abuse notation slightly, and for any Borel subset $\Omega\subseteq M(K)$, we denote the subspace-topology induced by $\tau_{_{AP(K)}}$  on $\Omega$, simply as $\tau_{_{AP(K)}}$. For the rest of this article, we consider (semihypergroup) action in the form of \Cref{action1}, which as discussed before, is equivalent to the standard definition (\Cref{action0}) of a semihypergroup action. Given a semihypergroup action $\pi$ of $K$ on a separated locally convex space $(E, Q)$ and any subset $\Omega \subset M(K)$, we say that a subset $F\subset E$ is $\Omega$-invariant, if we have that $\pi(\mu, x) = \pi_\mu(x) \in F$ for each $\mu \in \Omega, x\in F$. We say that the action $\pi$ has Property $(FP)$ if it satisfies the following property:
\begin{eqnarray*}
	(FP): & & \mbox{ For any compact, convex, } P(K)\mbox{-invariant subset } F \mbox{ of } E, \mbox{ if the induced}\\
	& &  \mbox{   action  } \tilde{\pi}:=\pi{|}_{_{P(K)\times F}}: P(K)\times F \ra F \mbox{ is separately } \tau_{_{AP(K)}} \mbox{-continuous,}\\
	& & \mbox{ then } \tilde{\pi} \mbox{ has a fixed point.}
\end{eqnarray*}

We show that the fixed-point property $(FP)$ completely characterizes the existence of a LIM on the space $AP(K)$ of almost periodic functions on any semitopological semihypergroup $K$.

\begin{theorem}\label{mainf}
Let $K$ be a semitopological semihypergroup. Then $AP(K)$ has a LIM if and only if  any semihypergroup action $\pi$ of $K$ on a separated locally convex space $(E, Q)$ satisfies Property $(FP)$. 
\end{theorem}

\begin{proof}
First, let $m$ be a left-invariant mean on $AP(K)$, and $\pi: M(K)\times E \ra E$ be a semihypergroup action of $K$ on a separated locally convex space $(E, Q)$. Further, let $F$ be a compact convex subset of $E$ such that $E$ is $P(K)$-invariant under $\pi$ such that the induced action $\tilde{\pi}:=\pi{|}_{_{P(K)\times F}}: P(K)\times F \ra F $ is separately $\tau_{_{AP(K)}}$-continuous.

Since $m$ is a LIM on $AP(K)$, it follows from \Cref{thmequiv} that there exists a net of measures $\{\mu_\alpha\}_{\alpha\in I}$ in $P(K)$ such that $\{(\mu * \mu_\alpha - \mu_\alpha)\}$ converges to $0$ in $(M(K),  \tau_{_{AP(K)}})$ for each $\mu \in P(K)$.

Pick some $x_0\in F$, and consider the net $\{\tilde{\pi}(\mu_\alpha, x_0)\}_{\alpha\in I} = \{\tilde{\pi}_{_{\mu_\alpha}}(x_0)\}_{\alpha\in I}$ in $F$. Since $F$ is compact, we will get a subnet $\{\tilde{\pi}_{_{\mu_\beta}}(x_0)\}_{\beta\in I'}$, $I'\subseteq I$, such that $\tilde{\pi}_{_{\mu_\beta}}(x_0) \ra z_0$ in $(E, Q)$ for some $z_0\in F$. Thus for each $\mu\in P(K)$ we have the following:
\begin{eqnarray*}
	\tilde{\pi}(\mu, z_0) &=& \tilde{\pi} \big(\mu,  \ \lim_\beta \tilde{\pi}(\mu_\beta, x_0)\big)\\
	&=& \lim_\beta \tilde{\pi} (\mu * \mu_\beta, x_0)\\
	&=& \lim_\beta \tilde{\pi} \big((\mu * \mu_\beta - \mu_\beta) + \mu_\beta, x_0)	\\
		&=& \lim_\beta \tilde{\pi} ((\mu * \mu_\beta - \mu_\beta), x_0) +  \lim_\beta \tilde{\pi} \big( \mu_\beta, x_0)\\
		&=&  \tilde{\pi} \big(\lim_\beta (\mu * \mu_\beta - \mu_\beta), x_0\big) +  \lim_\beta \tilde{\pi} \big( \mu_\beta, x_0)\\
		&=& \tilde{\pi}(0, x_0) + \lim_\beta \tilde{\pi}_{_{\mu_\beta}}(x_0) \ = \ z_0,
	\end{eqnarray*}
where the second and fifth equality holds since the action $\tilde{\pi}$ is $\tau_{_{AP(K)}}$-continuous, and the fourth  equality follows from the bilinearity (\Cref{action1}) of $\tilde{\pi}$. Hence in other words, we see that $\pi$ satisfies Property $(FP)$.

Conversely, assume that any semihypergroup action $\pi$ of $K$ on a separated locally convex space $(E, Q)$ satisfies Property $(FP)$. In particular, we set $E= AP(K)^*$, and equip it with the weak$^*$-topology. We define a map $\pi:M(K)\times AP(K)^* \ra AP(K)^*$ by $$\pi(\mu, \phi) = \pi_\mu(\phi) := \ep_\mu \star \phi,$$
for any $\mu \in M(K), \phi\in AP(K)^*$. First note that $\pi$ is bilinear since $(AP(K)^*, \star)$ is an algebra, and $\ep_{\mu + \nu} = \ep_\mu + \ep_\nu$ for any $\mu, \nu \in M(K)$ by construction. In addition, $\pi$ is indeed a semihypergroup action of $K$ on $AP(K)^*$ since for any $\mu, \nu\in M(K)$ and $\phi\in AP(K)^*$ we have that 
\begin{eqnarray*}
	 \pi_{_{(\mu*\nu)}} (\phi) &=& \ep_{_{(\mu*\nu)}} \star \phi \\
	 &=& (\ep_\mu \star \ep_\nu) \star \phi \\
	 &=& \ep_\mu \star (\ep_\nu \star \phi) \ = \  \pi_\mu(\ep_\nu\star \phi) \ = \ (\pi_\mu \circ \pi_\nu) (\phi),
\end{eqnarray*}
where the second and third  equalities  follow from \Cref{remw} and associativity of Arens product, respectively. Now, set $F:= \mathcal{M}(AP(K))$. Since $\mathcal{M}(AP(K))$ is weak$^*$-closed in $\mathcal{B}_1(AP(K)^*)$, we have that $F$ is a compact convex subset of $E$. Moreover,  for each $\mu\in P(K)$, $\omega\in F$, since $\pi_\mu \in L(E)$, we have that
$$||\pi_\mu(\omega)|| = ||\ep_\mu \star \omega|| \leq ||\ep_\mu|| \: ||\omega|| = ||\mu|| \ ||\omega|| = 1 $$
Since $AP(K)$ contains the constant function $1$, we have that $ \pi_\mu(\omega) (1) = 1 = ||\pi_\mu(\omega)||$, i.e, $F$ is $P(K)$-invariant under the action $\pi$. Now consider the induced action $\tilde{\pi}:=\pi{|}_{_{P(K)\times F}}: P(K)\times F \ra F$. Note that the action $\tilde{pi}$ coincides with the map $\Omega_{P(K)}$ defined in the proof of \Cref{thmequiv}, and hence is separately $\big{(} \tau_{_{AP(K)}} \times \mbox{ weak}^*\mbox{-topology}\big{)}- \mbox{ weak}^*\mbox{-topology}$ continuous. 

Since $\pi$ satisfies Property $(FP)$,  $\tilde{\pi}$ has a fixed point $ m\in F= \mathcal{M}(AP(K))$. Hence for each $\mu\in P(K)$ we have: $$ m= \tilde{\pi}_\mu(m) = \ep_\mu \star m .$$
Thus $m$ is a LIM on $AP(K)$ by \Cref{thmequiv}.
\end{proof}

%%%%%%%%%%%%%%%%%%%%%%%%%%%%%%%%%%%%%%%%%%%%%%%%%%%%%%%%%%%%%%%%%%%%%%%%%%%%%%%%%%%%%%%%%%%%%%%%%%%%%%%%%%%%%%%%%%%%%%%%%%%%%%%%%%%%%%%%%%%%%%%%
%%%%%%%%%%%%%%%%%%%%%%%%%%%%%%%%%%%%%%%      OPEN QUESTIONS AND FURTHER WORK       %%%%%%%%%%%%%%%%%%%%%%%%%%%%%%%%%%%%%%%%%%%%%%%%%%%%%%%%%%%%%
%%%%%%%%%%%%%%%%%%%%%%%%%%%%%%%%%%%%%%%%%%%%%%%%%%%%%%%%%%%%%%%%%%%%%%%%%%%%%%%%%%%%%%%%%%%%%%%%%%%%%%%%%%%%%%%%%%%%%%%%%%%%%%%%%%%%%%%%%%%%%%%%

\section*{Acknowledgement}

\noi The author would like to sincerely thank her doctoral thesis advisor  Dr. Anthony To-Ming Lau, for suggesting the topic of this study and for the helpful discussions during the initiation of this work. She would also like to gratefully acknowledge the financial support provided by the Indian Institute of Technology Kanpur, India, for the course of this work. %She is also grateful to the referee for the valuable comments and suggestions.

\end{document}